\documentclass[11pt]{article}
\usepackage{geometry}
\usepackage[parfill]{parskip}
\usepackage{graphicx}

\usepackage{amssymb}

\usepackage{epstopdf}

\usepackage{amsmath,amsthm,amscd,amssymb}
\usepackage{latexsym}
\numberwithin{equation}{section}

\theoremstyle{plain}
\newtheorem{theorem}{Theorem}[section]
\newtheorem{lemma}[theorem]{Lemma}

\newtheorem{proposition}[theorem]{Proposition}

\theoremstyle{definition}
\newtheorem{definition}[theorem]{Definition}

\theoremstyle{remark}

\newtheorem{claim}[theorem]{Claim}

\newtheorem{case[theorem]}{Case}

\newcommand{\fn}{F _3 ^N}
\newcommand{\beqa}{\begin{eqnarray*}}
\newcommand{\eeqa}{\end{eqnarray*}}
\newcommand{\beqan}{\begin{eqnarray}}
\newcommand{\eeqan}{\end{eqnarray}}
\newcommand{\one }{\mathbf 1 }

\newcommand{\calf}{\mathcal F}
\newcommand{\calp}{\mathcal P}
\newcommand{\calq}{\mathcal Q}

\title{Structure in additively nonsmoothing sets}
\author{Michael Bateman and Nets Hawk Katz}

\date{}

\begin{document}

\maketitle

\begin{abstract}
Sets with many additive quadruples are guaranteed to have many additive octuples, by H\"{o}lder's inequality.  Sets with not many more than this are said to be additively nonsmoothing.  We give a new proof of a structural theorem for nonsmoothing sets that originally appeared in work of the authors (\cite{BK}) on the size of cap sets in $\fn $.
\end{abstract}

\section{Introduction}
In this paper we reprove a structural theorem from \cite{BK} for sets that are not \it additively smoothing\rm.  The notion of additive smoothing was introduced in a recent paper by the authors, where the spectra of large cap sets (i.e., sets in $\fn$ without any lines) are shown to be additively nonsmoothing.  See \cite{BK}.
We begin by reviewing several definitions, including that of additively smoothing.  
The setting for this paper is an abelian group $Z$.  
\begin{definition}
For a set $A\subseteq Z $, and $m=1, 2,3, \dots$, we define the \it additive energies \rm of $A$ by  
\beqa
E_{2m} (A) = | \{ (a_1, \dots , a_{2m} ) \in A^{2m} \colon a_1 + \dots + a_m = a_{m+1} + \dots  + a_{2m} \} |.
\eeqa
\end{definition}
The quantity $E_4 (A)$ is typically called the additive energy of $A$.  The importance of the higher order energies is made clear in 
\cite{BK}, although the theorem here uses only $E_4$ and $E_8$.  
\begin{definition}
We say a set $A$ is $\sigma$-smoothing if 
\beqa
E_8 (A) \sim |A|^{\sigma} { E_4 (A) ^3 \over |A| ^2 }.
\eeqa
\end{definition}
When we casually write that a set is ``nonsmoothing", we mean that it is $\sigma$-smoothing for a small value of $\sigma$; so for example, a set with exactly $E_8 (A) = { E_4 (A) ^3 \over |A| ^2 }$ is $0$-smoothing.  This definition measures the sharpness of the first inequality in Proposition \ref{holder} below.  We state the main theorem already, but encourage readers unfamiliar with the notion of additive smoothing to skip to Section \ref{examplesection} for some examples.
In this paper we prove the following structural theorem about sets with minimal additive smoothing.  It essentially appeared in 
\cite{BK} as Theorem 6.10.  The significant new ingredient here is the notion of sideways comity, which allows us to avoid some of the technicalities in the proof in \cite{BK}.  On the other hand, the function $f$ here gives much worse dependence on $\sigma$ than the function $f$ from \cite{BK}. 

%For our application to the spectrum of a cap set, the reader should have in mind that $\tau = 1 - O (\epsilon) $ and $\sigma = O(\epsilon)$.

\begin{theorem} \label{main}
Fix $\tau _0 >0$.
There exists a function $f_{\tau _0} \colon (0, 1 ) \rightarrow (0, \infty) $ with $f_{\tau _0}(\eta ) \rightarrow 0$ as $\eta \rightarrow 0 $ such that the following holds.
Let $\Delta \subseteq Z$ be a symmetric set (i.e., $\Delta = -\Delta$) of size $M$.  Let $\sigma_0 >0 $.  Assume that $E_4 (\Delta ' ) \sim M^{2+ \tau_0 }$ for every 
$\Delta ' \subseteq \Delta $ with $|\Delta '| \gtrsim |\Delta |$, and that $\Delta $ is at most $\sigma _0$-smoothing, i.e., 
$E_8 (\Delta ) \lesssim M^{4 + 3 \tau_0 + \sigma_0}$.  Then there exists $\alpha \geq 0$ such that for 
$j = 1, 2, \dots, M^{\alpha - f_{\tau _0}(\sigma)} $, we have sets $H_j \subseteq Z$, sets  $X _j \subseteq Z$,  and $B _j \subseteq \Delta $ such that 
%\beqa 
%|B| \gtrsim M^{3 - \alpha - f(\sigma)} ,
%\eeqa
\beqa
|H_j| \lesssim M^{\tau + \alpha + f_{\tau _0}(\sigma _0)},
\eeqa
\beqa
|X_j| \lesssim M^{ 1-\tau - 2 \alpha + f_{\tau _0}(\sigma _0)},
\eeqa
\beqa
|H_j-H_j| \lesssim |H_j|^{1+ f_{\tau _0}(\sigma _0)},
\eeqa
such that 
\beqa
|(X_j+H_j) \cap B_j | \gtrsim M^{ 1 - \alpha  - f_{\tau _0}(\sigma _0) },
\eeqa
and such that $B_k \cap B_j = \emptyset $ unless $k=j$.
\end{theorem}
We remark that as a consequence of the estimates on $|(X_j + H_j) \cap \Delta |$, we also have lower bounds on $|H_j|$, and 
$|X_j|$.  Further, by applying Freiman's theorem, one can conclude that the set $H$ is efficiently contained in a subspace or a coset progression, with the details depending on the specific setting $Z$.

We take a moment to state the asymmetric Balog-Szemeredi-Gowers theorem.  This will help us find subsets with good additive properties in sets with good comity and sideways comity.  (Both of these terms will be defined below.)  
\begin{lemma} \label{absg} Let $B,C \subset Z$ be such that there are at least $|B|^{1-\eta} |C|^2$ additive
quadruples of the form
$$b_1+c_1=b_2+c_2$$
with $b_1,b_2 \in B$ and $c_1,c_2 \in C$. Then there 
exists $\mu = \mu (\eta, {|B|\over |C|})$, with $\mu (\eta , {|B|\over |C|})  \rightarrow 0$ as $\eta \rightarrow 0$, and there exist 
 $K \subset Z$ and 
$X \subset Z$ with 
$$|X| \lesssim  |B|^{\mu} {|B| \over |C|},$$
so that
$$|B \cap (X+K)| \gtrsim |B|^{1-\mu},$$
\beqa
|K-K| \lesssim |K|^{1+ \mu} 
\eeqa
and there exists an element $x \in Z$ so that
$$|C \cap (x+K)| \gtrsim |C|^{1-\mu} .$$
In particular, the last inequality implies
$$|K| \gtrsim |C|^{1-\mu}.$$
\end{lemma}
See \cite{TV} Theorem 2.35 for a proof.
%An only slightly stronger form of this lemma appears in the book of Tao and Vu as Lemma 2.35.  \cite{TV06}  (They account for more general $L$.)  
%Another way of stating this result which we will use frequently is to define a function $f$ with
%$\lim_{t \longrightarrow 0} f(t)=0$ and to let $\mu=f(\eta)$. We will use this kind of notation
%frequently in the paper with the choice of $f$ varying from line to line.

%
%This can be combined with Freiman's theorem to show that the set $K$ is contained in a subspace of dimension $N^{O(\mu)}$:
%
%\begin{theorem} \label{Freiman} A $\mu$ additively closed set is contained in a subspace of dimension
%$N^{O(\mu)}$. \end{theorem}

{\bf Acknowledgements}  The first author is supported by an NSF postdoctoral fellowship, DMS-0902490.  The second author
is partially supported by NSF grant DMS-1001607.

%%%%%%%%%%%%%%%%%%%%%%%%%%%%%%%%%%%%%%%%%%%%%%%%%%%%%%%%%%%%%%%%%%%%%%%%%%%%%%%%%%%

\section{Examples} \label{examplesection}

%%%%%%%%%%%%%%%%%%%%%%%%%%%%%%%%%%%%%%%%%%%%%%%%%%%%%%%%%%%%%%%%%%%%%%%%%%%%%%%%%%%

We give a quick corollary of H\"{o}lder's inequality that motivates the definition of additively nonsmoothing.
\begin{proposition} \label{holder}
If $Z$ is finite, then for any set $A$, we have 
\beqa
E_8 (A) \geq { E_4 (A) ^3 \over |A|^2 }.
\eeqa
Further, we have 
\beqa
E_8 (A) \leq |A|^4 E_4 (A).
\eeqa
\end{proposition}
\begin{proof}
A straightforward calculation establishes the identity 
\beqa
E_{2m}(A) = |Z|^{2m-1}  \sum _{\xi \in Z} |\widehat{ \one _A } (\xi )|^{2m} 
\eeqa
for $m=1, 2, \dots $.  (Here $\widehat{\one _A}$ is the Fourier transform of $\one _A$, 
\beqa
\widehat{f} (\xi) = {1\over |Z|} \sum _{x \in Z} f(x) e^{2\pi i \langle \xi, x \rangle }
\eeqa
and $\langle \cdot , \cdot \cdot \rangle $ is a nondegenerate symmetric bilinear form.  See \cite{TV} Section 4.1 for details in our general setting.)
When $m=1$, this gives us 
\beqa
|A| = E_2 (A) = |Z| \sum _{\xi \in Z} |\widehat{ \one _A } (\xi )|^{2} 
\eeqa
which is just Plancherel's equality for the function $\one _A$.  H\"{o}lder's inequality yields 
\beqa
E_4 (A) & = & |Z|^3 \sum _{\xi \in Z} |\widehat{ \one _A } (\xi )|^{4}  \\
& \leq &  \left( |Z| \sum _{\xi \in Z} |\widehat{ \one _A } (\xi )|^{2} \right) ^{2\over 3} 
	  \left( |Z|^7 \sum _{\xi \in Z} |\widehat{ \one _A } (\xi )|^{8} \right) ^{1\over 3} \\
& = & |A| ^{2 \over 3} E_8 (A) ^{1\over 3}.
\eeqa
This proves the first claim.  To prove the second claim, just note that 
\beqa
E_8 (A) &=& |Z|^7 \sum_{\xi \in Z} |\widehat{\one _A} (\xi)|^8 \\
&\leq & |Z|^7 \sup_{\xi \in Z} |\widehat{\one _A} (\xi) |^4  \sum_{\xi \in Z} |\widehat{\one_A} (\xi)|^4 \\
& \leq &  |A|^4 |Z|^3 \sum_{\xi \in Z} |\widehat{\one _A} (\xi)|^4 \\
&=& |A|^4 E_4 (A),
\eeqa
since $|\widehat{\one _A}(\xi)| \leq {|A| \over |Z|}$ for any $\xi$.
\end{proof}
As examples of the two extremes, consider a ``random" set $A$ of size $N$ in a subgroup $H$ of size $N^{1 + \epsilon}$.  Given $a_1, a_2, a_3 \in A$, we know 
\beqa
a_1 + a_2 - a_3 \in H;
\eeqa 
further 
\beqa
a_1 + a_2 - a_3 \in A
\eeqa
with probability $N^{-\epsilon} = |A| ^{-\epsilon}$, since $|A| = |H| N^{-\epsilon}$.  Hence we expect $E_4(A) \sim |A|^{3 -\epsilon}$.  By a similar calculation we expect $E_8 (A) \sim |A|^{7 - \epsilon}$.  Note that this example achieves the maximal $E_8$ allowed by the proposition above.  On the other hand, if we let $A$ be given by $H+R$ where $H$ is a subgroup of size $N^{1-\epsilon}$ and $R$ is a ``random" set of size $N^{\epsilon}$, then 
\beqa
E_4(A) = E_4(H)E_4(R) = N^{3 -3\epsilon } N^{2\epsilon} = N ^{3-\epsilon}.
\eeqa
However in this case
\beqa
E_8(A) = E_8(H)E_8(R) = N^{7-7\epsilon} N^{4 \epsilon } = N^{7-3\epsilon}.
\eeqa
Similarly if $A$ is the union of (unrelated) subspaces $H_j$, for $j= 1, 2, \dots, N^{\epsilon \over 2}$ where $|H_j| = N^{1-{\epsilon \over 2} }$ for each $j$, then 
\beqa
E_4(A) \sim \sum _{j =1 } ^{N^{\epsilon \over 2} } E_4(H_j) = N^{\epsilon \over 2} N^{3-{3\epsilon \over 2} } = N^{3-\epsilon}  
\eeqa
and 
\beqa
E_8(A) \sim \sum _{j =1 } ^{N^{\epsilon \over 2} } E_8(H_j) = N^{\epsilon \over 2} N^{7-{7\epsilon \over 2} } = N^{7-3\epsilon}.  
\eeqa
Note that these last two sets achieve the minimal $E_8$ allowed by the proposition above.

%%%%%%%%%%%%%%%%%%%%%%%%%%%%%%%%%%%%%%%%%%%%%%%%%%%%%%%%%%%%%%%%%%%%%%%%%%

\section{A simple reduction}

%%%%%%%%%%%%%%%%%%%%%%%%%%%%%%%%%%%%%%%%%%%%%%%%%%%%%%%%%%%%%%%%%%%%%%%%%%

The bulk of the work in this paper goes toward proving the following theorem, which identifies a large piece of the set $\Delta$ with substantial structure.  Theorem \ref{main} then follows by repeatedly finding these large pieces until most of $\Delta$ has been exhausted.  
\begin{theorem}\label{specstruct}
Fix $\tau _0> 0$.
There exists a universal function $f_{\tau _0} \colon (0, 1 ) \rightarrow (0, \infty) $ with $f_{\tau _0}(\eta ) \rightarrow 0$ as $\eta \rightarrow 0 $ such that the following holds.
Let $\Delta \subseteq Z$ be a symmetric set of size $M$.  Let $\sigma_0 >0 $ be such that $E_4 (\Delta  ) \sim M^{2+ \tau_0 }$ 
%for every $\Delta ' \subseteq \Delta $ with $|\Delta '| \gtrsim |\Delta |$, 
and such that $\Delta $ is at most $\sigma _0 $-smoothing, i.e., 
$E_8 (\Delta ) \lesssim M^{4 + 3 \tau_0 + \sigma_0}$.  Also assume that for every $a\in \Delta$, 
\beqa
|\{ (b,c,d) \in \Delta ^3 \colon a-b =c-d \}| \lesssim M^{1+\tau}.
\eeqa
Then there exists $\alpha \geq 0$, a symmetric set $H \subseteq Z$, and a symmetric set $X \subseteq Z$ such that 
%\beqa 
%|B| \gtrsim M^{3 - \alpha - f(\sigma)} ,
%\eeqa
\beqa
|H| \lesssim M^{\tau + \alpha + f_{\tau _0} (\sigma_0)},
\eeqa
\beqa
|X| \lesssim M^{ 1-\tau - 2 \alpha + f_{\tau _0} (\sigma _0)},
\eeqa
\beqa
|H-H| \lesssim |H|^{1+ f_{\tau _0} (\sigma _0)},
\eeqa
and such that 
\beqa
|(X+H) \cap \Delta | \gtrsim M^{ 1 - \alpha  - f_{\tau _0} (\sigma _0) }.
\eeqa
\end{theorem}
We remark that the symmetry conclusions on $H$ and $X$ are in place only to guarantee that after removing $X+H$ from $\Delta$, the remainder is still symmetric.

\begin{proof}[Proof of Theorem \ref{main} given Theorem \ref{specstruct}]
Our first fact allows us to assume that no $a\in \Delta$ participates in too many quadruples, which is one of the hypotheses needed for Theorem \ref{specstruct}.
\begin{proposition} \label{nottoopopular}
If $E_4(\Delta ') \gtrsim M^{2+ \tau}$ for every $\Delta ' \subseteq \Delta $ with $|\Delta '| \gtrsim |\Delta |$, then there is 
$\widetilde{\Delta}\subseteq \Delta $ with $E_4 ( \widetilde{\Delta} ) \gtrsim M^{2+ \tau}$ such that for each $a\in \widetilde{\Delta}$, 
\beqa
|\{ (b,c,d) \in \Delta ^3 \colon a= b+c - d \}| \lesssim M^{1 + \tau}  .
\eeqa
\end{proposition}

In other words, no $a$ participates in more than $\sim M^{1 + \tau } $ quadruples.  
\begin{proof}
Observe that
\beqa
\sum _{a\in \Delta } |\{ (b,c,d) \in \Delta ^3 \colon a= b+c - d \}| \lesssim M^{2 + \tau },
\eeqa
and hence there are fewer than $\sim {1\over C} M$ elements $a$ such that the summand is $\geq C M^{1 + \tau }$.  We simply remove this set of $a$ and note that the remaining set, which we call $\widetilde{\Delta}$, still has essentially full energy by hypothesis since it contains most elements of $\Delta$.
\end{proof}

Find $\widetilde{\Delta}$ satisfying the conclusion of Proposition \ref{nottoopopular} above.  Importantly, the sets $\Delta _j$ defined below inherit this property (so we do not need to apply Proposition \ref{nottoopopular} more than once).  Now we may apply Theorem \ref{specstruct} to find $ \alpha _1, B_1, H_1, X_1$.  Then let $\Delta _1 = \widetilde{\Delta} \setminus B_1 $.  Note that since 
$| \Delta _1 | \gtrsim |\Delta| $, $\Delta _1$ still has essentially full energy by the hypothesis of Theorem \ref{main}, and hence satisfies the hypotheses of Theorem \ref{specstruct}.  (The symmetry hypothesis is also satisfied, as mentioned immediately after the statement of Theorem \ref{specstruct}.)   Having defined 
$\Delta _{j-1}$, apply Theorem \ref{specstruct} to find $ \alpha _j, B_j, H_j, X_j$, then define 
$B_j = (X_j + H_j) \cap \Delta $ and $\Delta _{j} = \Delta _{j-1} \setminus B_j$.  We may continue to find blocks $B_j$ until 
\beqa
\left| \bigcup _{k=1} ^{j-1} B_k \right| \gtrsim |\Delta |.
\eeqa
Not all the $\alpha _j$ need to be equal, but we fix this by pigeonholing to find $\alpha $ such that 
\beqa
\left| \bigcup _{k \colon |B_k| \sim N^{1-\alpha \pm f(\sigma) } } B_k \right| \gtrsim { |\Delta | \over \log M }.
\eeqa
\end{proof}

An outline of the proof of Theorem \ref{specstruct} is as follows.  First, we pigeonhole to find $D \subseteq \Delta - \Delta $ such that 
$|\Delta \cap (x+ \Delta )| $ is approximately constant for $x\in D$ and such that differences in $D$ account for most of the energy in $\Delta$.  
$D$ corresponds to the differences from a graph $G\subseteq \Delta \times \Delta $.  We will measure how elements of $D$ interact with each other and with elements of $\Delta$ using quantities called \it comity\rm, which was introduced in \cite{BK} (and even to some degree in \cite{KK}), and \it sideways comity\rm, which we introduce here.  When both of these quantities are small, we can make precise statements about the structure of 
$\Delta$ by using the asymmetric Balog-Szemeredi-Gowers theorem above.  The exact structure depends on $|G|$.  
See Section \ref{structuresection} for details on finding this structure.  When either of these quantities is large, we may find a graph $G'$ with 
$|G'| >> |G|$ such that $G'$ still accounts for most of the energy of $\Delta$.  See Section \ref{comitysection} for the large comity case.  See Section \ref{sidewayscomitysection} for the large sideways comity case.  This process terminates  once we reach $|G'| \sim |\Delta |^2 $, which happens after a controlled number of iterations.  By this point, we must have achieved small comity and small sideways comity.  See Section \ref{iterationsection} for details about the iteration.

%%%%%%%%%%%%%%%%%%%%%%%%%%%%%%%%%%%%%%%%%%%%%%%%%%%%%%%%%%%%%%%%%%%%%%%%%%%

\section{Additive structures}

%%%%%%%%%%%%%%%%%%%%%%%%%%%%%%%%%%%%%%%%%%%%%%%%%%%%%%%%%%%%%%%%%%%%%%%%%%%

In this section we present some basic definitions.
\begin{definition}
We define an additive structure $\alpha$ on $\Delta$ at height $\alpha$ to be a  pair $(G,D)$, where $G \subseteq \Delta \times \Delta $ is a graph such that $|G| \sim M^{2- \alpha }$, where $D$ is a set such that $a-b \in D$ for $(a,b) \in G$, and where 
$|\Delta \cap (a-b + \Delta )| $ is essentially constant for $(a,b) \in G$, i.e., 
\beqa
\sup_{(a,b)\in G} |\Delta \cap (a-b + \Delta )| \leq 2 \min_{(a,b)\in G} |\Delta \cap (a-b + \Delta )|  .
\eeqa
\end{definition}
\begin{definition}
For any graph $G$, we define the energy of $G$:
\beqa
E(G) = \sum _x |\{ (a,b) \in G \colon a-b = x \} | ^2 .
\eeqa
\end{definition}
Note that this is just the number of quadruples in $\Delta$ accounted for by pairs in the graph.  
The following proposition shows that we can find an additive structure capturing most of the energy of $\Delta$.  This will help us start the iteration discussed in Section \ref{iterationsection}.
\begin{proposition} \label{pigeonhole}
There exists an additive structure $(G, D)$ at height $\alpha$ for some $\alpha \leq {1-\tau \over 2} $ such that 
\beqa
E(G) \gtrsim{{ M^{2  + \tau } } \over { (\log M)^2 } } .
\eeqa
\end{proposition}
\begin{proof}
To see this, just note that 
\beqa
E_4 (\Delta ) = \sum _x |\Delta \cap (x+ \Delta ) | ^2 = \sum _{x\in \Delta - \Delta}  |\Delta \cap (x+ \Delta ) | ^2.
\eeqa
Since $0\leq |\Delta \cap (x+ \Delta ) | \leq |\Delta |$, we can pigeonhole over $ \log M$ scales to find a set 
$D \subseteq \Delta - \Delta $ such that 
\beqa
\sum _{x \in D} |\Delta \cap (x+ \Delta ) | ^2 \gtrsim  {{ M^{2  + \tau } } \over {  \log M } } 
\eeqa
and such that $ |\Delta \cap (x+ \Delta ) | \sim M^{\alpha + \tau } $ for some $\alpha \geq 0$ and every  $x\in D$.  Then define 
\beqa
G = \{ (a,b) \in \Delta ^2 \colon a-b \in D \} .
\eeqa
This pair $(G,D)$ is an additive structure at height $\alpha$.  We now show that $\alpha $ can be taken $\leq {1-\tau \over 2}$.  Note that 
\beqa
{ M^{2+ \tau} \over \log M }  \lesssim E(G) = 
	\sum _{a,c} | \{ (b,d) \colon (a,b)\in G \text{ and } a-b=c-d \} |.
\eeqa
We know that for each $a\in \Delta $ there are at most $\sim M^{1-\alpha}$ many $b \in \Delta $ such that 
$(a,b)\in G$ (for otherwise we would violate the assumption of Theorem \ref{specstruct}).  Hence the summand on the right is bounded by 
$\sim M^{1-\alpha}$.  This implies that the summand is nonzero for a set $|G'|$ of pairs $(a,c)$, with 
\beqa
|G'| \gtrsim { M^{2+ \tau} \over \log M } {1\over M^{1-\alpha} }  = {  M^{ 1+ \tau + \alpha }  \over \log M} ,
\eeqa
and hence (after pigeonholing over subgraphs of $G'$ such that 
$| \Delta \cap (a-c + \Delta) |$ is essentially constant, which gives us the corresponding $D'$) 
that at least ${ M^{2+ \tau} \over (\log M )^2 } $ of the quadruples in $\Delta$ come from a graph of height 
$\alpha ^{\prime}$ with $\alpha ^{\prime} \leq 2-(1+ \tau + \alpha) = 1-\tau - \alpha $.  Note that $1-\tau - \alpha$ decreases as $\alpha $ increases, and they are equal when $\alpha = {1-\tau \over 2}$.  This proves the claim about the height, since either 
$\alpha $ or $\alpha '$ is  $ \leq {1-\tau \over 2}$.
\end{proof}
%Now we simply iterate Lemma \ref{iteration}, finding a sequence of additive structures 
%\beqa
%(G_1, D_1), (G_2, D_2), \dots , (G_K, D_K)
%\eeqa
%such that the height of $(G_j , D_j)$ is less than $1 + O(f(\sigma)) - j {C\over {\log {1\over \sigma } } } $, and such that 
%\beqa
%E(G_j) \gtrsim M^{6+ \tau} M^{-O(\sigma 5 ^j) } .
%\eeqa
%It is clear that since the height decreases by ${C\over {\log {1\over \sigma } } }$ at each iteration, we have 
%$K \leq  {1\over C} \log { 1\over \sigma }  $.  In particular, we will obtain the structural result in the statement of the lemma, starting with 
%$\delta \geq M^{-O(\sigma 5 ^{{1\over C} \log { 1\over \sigma }   }  )}$.  Taking $C$ large  is enough to ensure that this value of $\delta$ goes to $1$ as $\sigma $ approaches $0$.  This yields the conclusion of Theorem \ref{specstruct}.

%%%%%%%%%%%%%%%%%%%%%%%%%%%%%%%%%%%%%%%%%%%%%%%%%%%%%%%%%%%%%%%%%%%

\section{Comity} \label{comitysection}

%%%%%%%%%%%%%%%%%%%%%%%%%%%%%%%%%%%%%%%%%%%%%%%%%%%%%%%%%%%%%%%%%%%
The goal of this section is to introduce the notion of \it comity \rm  and to prove Lemma \ref{comitytonight}, which tells us that either an additive structure has good comity, or the set $\Delta$ admits an additive structure of lower height.  Both the notion of comity and 
Lemma \ref{comitytonight} appeared in \cite{BK}.
We start by introducing a convenient shorthand.  For $x\in \Delta - \Delta$, define
\beqa
\Delta [x] = \Delta \cap (x+ \Delta) = \{ a \in \Delta \colon a-x \in \Delta \};
\eeqa
i.e., $\Delta [x]$ is the set of elements that participate (in the first position) in a difference of $x$.  To define comity, assume we are have an additive structure $(G,D)$ at height $\alpha$.  
By interchanging sums and applying Cauchy-Schwarz we have
\beqa
\sum _{x\in D} \sum _{y\in D} | \Delta [x] \cap \Delta [y] |
&=& \sum _{a\in \Delta } \left( \sum _{x\in D} \one _{\Delta [x]} (a) \right) ^2 \\
&\gtrsim &  M^{3-2 \alpha } .
\eeqa
By pigeonholing over $C \log M$ scales, we can find $\calp \subseteq D\times D$ and $\beta $ such that 
\beqan \label{needsaname}
\sum _{(x,y) \in \calp } | \Delta [x] \cap \Delta [y] | \gtrsim { { M^{3-2\alpha }  } \over { \log M} }, 
\eeqan
and such that $| \Delta [x] \cap \Delta [y] | \sim M^{\beta} $ for $(x,y) \in \calp $.  Note that this immediately implies 
$|\calp | \gtrsim  { { M^{3-2\alpha - \beta}  } \over {  \log M} } . $
We see below that paying attention to $\beta$ is profitable, which prompts the first of our key definitions.  Note that of course $\beta \leq \tau + \alpha $ since $|\Delta [x] | \sim  M^{\tau + \alpha} $.  
\begin{definition} \label{comitydef}
We say that an additive structure $(G,D)$ at height $\alpha $ has \it comity $\mu$ \rm if there exists 
$ \beta \geq \tau + \alpha - \mu $ and 
$\calp \subseteq D \times D$ such that $| \Delta [x] \cap \Delta [y] | \sim M^{\beta} $ for $(x,y) \in \calp $, and such that 
$|\calp | \gtrsim  { { M^{3-2\alpha - \beta}  } \over {  \log M} }$.  
\end{definition}
The computation above proves the following:
\begin{proposition} \label{comityexists}
For any additive structure $(G,D)$ at height $\alpha $, there exists $\mu >0$ such that $(G,D)$ has comity $\mu$.
\end{proposition}
We remark that by the definition, if $(G,D)$ has comity $\mu$, then it also has comity $\mu '$ for all $\mu ' \geq \mu$.  In 
Section \ref{iterationsection} we will need to select a particular value, but this is of no consequence.
We now prove that either our additive structure has small comity parameter, or there exists an additive structure at lower height.  The key assumption in this lemma is the hypothesis of small additive smoothing; in fact, this is the only part of the structural theorem that requires it.  
\begin{lemma} \label{comitytonight}
Let $(G,D)$ be an additive structure on $\Delta $ at height $\alpha$ such that $E(G) \gtrsim M^{2 + \tau}$,
and let $\mu >0$.  Assume $E_8 (\Delta ) \lesssim M^{4 + 3 \tau + \sigma}$.
Then either $(G,D)$ has comity $\mu$ or there exists an additive structure $(G', D')$ on $\Delta $ at height 
$\leq \alpha - \mu + 2\sigma$ such that $E(G') \gtrsim M^{-2\sigma } E(G) $.
\end{lemma}
We remark that in our application $\sigma$ will be much smaller than $\mu$, so the height will decrease by essentially $\mu$.  
\begin{proof}
First, find $\calp $ and $\beta$ as guaranteed by Proposition \ref{comityexists}.  Then define
\beqa
D_{\beta } = \{ d \in \Delta - \Delta \colon |\Delta \cap (d + \Delta )| \geq M^{\beta} \} |.
\eeqa
%We can use Chebyshev's inequality to estimate $|D_{\beta}|$:
%\beqa
%|D_{\beta}| & \leq & {1\over {M^{2\beta } } } \sum _d | \delta \cap (d + \Delta ) | ^2  \\
%&=& {1\over {M^{2\beta } } } E(G) \\
%&\leq & M^{6 + \tau - 2 \beta } .
%\eeqa
%We now use this to obtain a lower bound on $E_4 (D)$, and hence on $E_8 (\Delta)$.  
Note that for each $(x,y) \in \calp $ 
(where $\calp$ is obtained just as before the statement of the lemma) 
we have 
\beqa
|\Delta \cap (x-y + \Delta ) | \geq | \Delta [x] \cap \Delta [y] | \geq M^{\beta} .
\eeqa
In other words, there are at least $M^{\beta}$ ways in which $x-y$ can be written as a difference of pairs $(c,d)\in \Delta ^2 $, i.e., 
$x-y \in D_{\beta} $.  This is because if $a\in \Delta [x] \cap \Delta [y]$, then $a-x = c$, $a-y=d$ for some $c,d \in \Delta $, so $x-y=d-c$.  There are $\gtrsim M^{\beta}$ such $a$ giving us $\gtrsim $ pairs $(d,c)$ with $x-y=d-c$.  Hence we have, using Cauchy-Schwarz, 
\beqa
{ { M^{3-2\alpha - \beta}  } \over { \log M} } & \lesssim & |\calp | \\ 
& \leq & \sum _{x\in D_{\beta} } |D \cap (z + D) | \\ 
& \leq & \sqrt{ E_4 (D) |D_{\beta}| } %\\
%& \lesssim & \sqrt{ E_4 (D) M^{6 + \tau - 2\beta }  } .
\eeqa
%Combining this with the lower bound 
%\beqa
%|\calp | \geq \delta ^2 { { M^{9-2\alpha - \beta}  } \over { C \log M} }
%\eeqa
%gives us 
%\beqa
%E_4 (D) \geq {{\delta ^4} \over {(C \log M )^2 } } M^{12 - 4\alpha - \beta } 
%\eeqa
Since each $x\in D$ has $ M^{\alpha + \tau } $ representations, (i.e., $|\Delta \cap (x + \Delta) |\sim M^{\alpha + \tau}$, ) 
we also know that 
\beqa
E_4 (D) ( M^{\alpha + \tau } ) ^4 \lesssim E_8 (\Delta ) \lesssim M^{4 + 3\tau + \sigma } 
\eeqa
Hence 
\beqa
|D_{\beta} | & \gtrsim & {1 \over {E_4 (D)} }  {{1 } \over { ( \log M) ^2 } } M^{6 -4\alpha - 2\beta } \\
&\gtrsim & {{1} \over { ( \log M) ^2 } }  M^{2 + \tau - 2 \beta - \sigma }.
\eeqa
We are now ready to define the graph $(G', D') $ in the statement of the lemma.  If $ \beta \geq \tau + \alpha - \mu  $, then we already 
have comity $\mu$, by definition.  So assume $\beta < \tau + \alpha - \mu$.  Then let
\beqa
\tilde{G} = \{ (a,b) \in \Delta \times \Delta  \colon a-b \in D_{\beta} \}.
\eeqa
By the estimate above, we have 
\beqa
|\tilde{G}| & \gtrsim & |D_{\beta}| M^{\beta} \\ 
&\gtrsim & {{1 } \over { (C \log M) ^2 } }  M^{2 + \tau - \beta - \sigma } \\
&\gtrsim & M^{2+ \tau - \beta - {3\over 2} \sigma } \\
&\gtrsim & M^{2 - \alpha - {3\over 2} \sigma + \mu} ,
\eeqa
because $\beta < \tau + \alpha - \mu$; this is because we do not have comity $\mu$.  We are essentially done, but we must note that by definition of $D_{\beta}$, a pair 
$(a,b)\in \tilde{G}$ satisfies $|\Delta \cap (a-b -\Delta ) | \gtrsim M^{\beta}$, and the inequality goes in only one direction.  To obtain an additive structure, we want essential equality.  Nevertheless, this can be obtained by a further pigeonholing to find $G' \subseteq \tilde{G}$ such that $|G'| \gtrsim {|\tilde{G}| \over \log M} $.  Call the corresponding difference set $D'$.  The computation immediately above  proves the claim about the height.  Further, this same estimate proves the estimate 
$E(G') \gtrsim M^{2+ \tau - 2\sigma} $ since each pair $(a,b) \in G'$ satisfies $| \Delta \cap (a-b + \Delta ) | \gtrsim M^{\beta}$.
%(See the discussion immediately before proposition \ref{nottoopopular}.)
\end{proof}

%%%%%%%%%%%%%%%%%%%%%%%%%%%%%%%%%%%%%%%%%%%%%%%%%%%%%%%%%%%%%%%%%%%

\section{Sideways comity}  \label{sidewayscomitysection}

%%%%%%%%%%%%%%%%%%%%%%%%%%%%%%%%%%%%%%%%%%%%%%%%%%%%%%%%%%%%%%%%%%%
The goal of this section is to introduce the notion of \it sideways comity\rm and to prove Lemma \ref{SWcomitytonight}, which tells us that an additive structure with good comity either also has good sideways comity, or the set $\Delta$ admits an additive structure at lower height.

For the following discussion assume we have an additive structure $(G,D)$ at height $\alpha$ with comity $\mu$.  
We now give the second of our key definitions which is for another comity-like notion.  First, define 
\beqa
\calf _x = \{ y \in D \colon (x,y) \in \calp \},  
\eeqa
where $\calp$ is the set of pairs $(x,y)$ such that $|\Delta[x] \cap \Delta [y]| $ is large (close to $M^{\tau + \alpha - \mu}$), as in the definition of comity, Definition \ref{comitydef}.
We know that 
\beqa
|\calf_x | \lesssim M^{1-\alpha + \mu }
\eeqa
because otherwise there exists $a\in \Delta [x] $ such that $a$ participates in many  more than $M^{1 + \tau} $ quadruples (and we have no such $a$, by assumption in Theorem \ref{specstruct}).  
For each $x\in D$ and $a\in \Delta [x]$, define the related sets
\beqa
F_{x,a} &=& \{ b \in \Delta \colon b-a \in \calf _x \} \\
F_a &=& \{ b \in \Delta \colon b-a \in D \} .
\eeqa
The sets $F_{x,a}$ will be more important to us, but consider for a moment the sum
\beqa
\sum _{x\in D} \sum _{b\in \Delta } | F_b \cap \Delta [x]|.  
\eeqa
It is straightforward to show that this is $\gtrsim {1 \over \log M} M^{3-2\alpha}$ by interchanging the sums, just as with 
\beqa
\sum_{x\in D} \sum _{y\in D} |\Delta[x] \cap \Delta [y] |.
\eeqa
In fact, we can prove the following slightly refined estimate, with nothing more than interchanging sums:
\beqa
\sum _{x\in D} \sum _{b\in\Delta } | \Delta [x] \cap F_{x,b} | 
&=& \sum _{x\in D} \sum _{b\in\Delta } \sum _c \one _{\Delta [x]} (c) 
		\one _{b+\calf _x } (c) \\
&=& \sum _{x\in D}  \sum _c \one _{\Delta [x]} (c) 
	\sum _{y\in\calf_x} \one _{\Delta [y] } (c) \\
	& \gtrsim & {1 \over \log M} M^{3-2\alpha} 
\eeqa
since the second-to-last display is equal to the sum in estimate \ref{needsaname} above, which can be seen by interchanging the sums.  
By pigeonholing, we can find $\gamma $  and $\calq \subseteq D \times \Delta $ such that 
\beqan \label{secondsum}
\sum _{(x,b) \in \calq} | \Delta [x] \cap F_{x,b} | 
\gtrsim {1 \over (\log M)^2 } M^{3-2\alpha} 
\eeqan
and such that 
\beqa
| \Delta [x] \cap F_{x,b} |  \sim M^{\gamma}
\eeqa
for $(x,b) \in \calq$.  Note that this implies $|\calq| \gtrsim {1 \over (\log M)^2 } M^{3-2\alpha - \gamma} $.  All of this discussion motivates the following definition.
\begin{definition}
We say an additive structure $(G,D)$ at height $\alpha$ with comity $\mu$ has \it sideways comity \rm $\nu$ if there exist
$\gamma \geq \tau+\alpha - \nu $  and $\calq \subseteq D \times \Delta $ such that 
$| \Delta [x] \cap F_{x,b} |  \sim M^{\gamma}$ for $(x,b) \in \calq$, 
and such that 
$|\calq| \gtrsim {1 \over (\log M)^2 } M^{3-2\alpha - \gamma}$.
\end{definition}
The computation above proves the following:
\begin{proposition}\label{SWcomityexists}
For any additive structure $(G,D)$ at height $\alpha $ and comity $\mu$, there exists $\nu >0$ such that $(G,D)$ has sideways comity $\nu$.
\end{proposition}
Note that 
$|\Delta [x] | \sim M^{\tau+\alpha}$; hence by the definition, having sideways comity $\nu$ requires 
\beqa
M^{\tau + \alpha } \sim |\Delta [x] | \lesssim |F_{x,b}| \sim N^{1-\alpha + \mu};
\eeqa
in other words, we need $\alpha \leq {1- \tau \over 2} + O(\nu + \mu )$.  Fortunately this is guaranteed by Proposition \ref{pigeonhole}.
\begin{lemma}\label{SWcomitytonight}
Suppose $( G, D)$ is an additive structure at height $\alpha $ with comity $\mu$ such that $E(G) \gtrsim  M^{2+ \tau}$.  Then either the structure has sideways comity $\nu$ or there exists an additive structure 
$( G' , D' ) $ of height $\leq \alpha + \mu - {\nu \over 2}$ such that $E(G') \gtrsim E(G) M^{-O(\mu)}$.  
\end{lemma}
We remark that in our application $\mu$ will be much smaller than $\nu$, so the height will decrease by essentially $\nu$.  
%Before we prove this lemma, let's see what we can do with sideways comity.  
%
%So now we have seen that sideways comity (together with comity) is enough to ensure substantial structure in the spectrum.  It remains to show that without sideways comity we can find an additive structure of lower height, which is the content of Lemma \ref{SWcomitytonight}.  We now turn our attention to its proof. 
\begin{proof} [Proof of Lemma \ref{SWcomitytonight}]
We begin by considering a pair $(x,b)\in \calq$.  Recall that for each such pair we have 
\beqan \label{condition}
| \Delta [x] \cap F_{x,b} | \sim M^{\gamma }.
\eeqan
We now prove the following claim:
\begin{claim}\label{sums}
For each  $(x,b)\in \calq$, we have 
\beqa
| \{ a \in \Delta [x] \colon | (a+b - \Delta ) \cap \Delta | 
		\gtrsim M^{\gamma - \mu } \} |  
				\gtrsim |\Delta [x] |  M^{-\mu}.
\eeqa
\end{claim}
\begin{proof} [Proof of Claim]
The condition \eqref{condition} immediately above tells us there are $M^{\gamma}$ many $c\in \Delta [x] $ such that $ c-b =y $ for some $y\in \calf _x $.  For each such $c$, there are at least $ M^{\tau+\alpha -\mu }$ many $a\in \Delta [x] $ such that 
$a-y\in \Delta $ (because $y$ is in $\calf _x$, and because we have $\mu$-comity).  Summing over $c \in \Delta [x]$ gives us $M^{\gamma + \tau + \alpha - \mu }$ quadruples $a+b = c+d$ with $a\in \Delta [x]$, $b$ fixed, $c,d \in \Delta$, and any given $a\in \Delta [x]$ appearing no more than $M^{\gamma }$ times.  Hence there is a set $\Delta _{x,b} \subseteq \Delta [x] $ of size $M^{\tau+\alpha - \mu }$ such that 
\beqa
|\Delta \cap (a+b - \Delta )| \gtrsim M^{\gamma - \mu }
\eeqa
for each $a \in \Delta _{x,b}$.  This is precisely what we claimed.  
\end{proof}
It remains to construct the graph $G'$ claimed in the lemma.  The set $D'$ will be contained in the set of differences $x$ such that 
$|\Delta \cap ( x+ \Delta ) | \gtrsim M^{\gamma - \mu }$, with an application of the pigeonhole principle required again, as in 
Lemma \ref{comitytonight}.  It is worth noting that we will actually show that there are lots of pairs whose \it sum \rm is in $D'$; by symmetry of $\Delta$ we can conclude that the are the same number of pairs whose \it difference \rm is in $D'$.  
For $b\in \Delta $ define 
\beqa
K_b = \{ x \colon (x,b) \in \calq \}.
\eeqa
We know from the definition of $\calq$ that 
\beqan \label{Ksizes}
\sum _b |K_b| \gtrsim |\calq | \gtrsim {1 \over (\log M)^2 } M^{3-2\alpha - \gamma }.
\eeqan
Because of the claim, we know that every $a \in \cup _{x \in K_b} \Delta _{x,b} $ satisfies 
\beqa
|\Delta \cap (a+b - \Delta )| \gtrsim M^{\gamma - \mu }.
\eeqa
Our goal is to show
\beqan \label{bigarea}
\left| \bigcup _{x \in K_b} \Delta _{x,b} \right| 
		\gtrsim |K_b | M^{\tau -1 + 2 \alpha  }.
\eeqan
Assuming \eqref{bigarea}, we are finally ready to define the graph $G'$ claimed in the statement of the lemma.  Let $\tilde{G}$ be the set of all pairs $(b,a) \in \Delta \times \Delta $ such that 
\beqa
 a\in \bigcup _{x \in K_b} \Delta _{x,b} .
\eeqa
Hence by estimates \eqref{bigarea} and \eqref{Ksizes}, we have 
\beqa
|\tilde{G}| &=& \sum _{b\in\Delta}  \left| \bigcup _{x \in K_b} \Delta _{x,b}  \right| \\
		& \gtrsim & \sum _{b\in\Delta} |K_b | M^{\tau - 1 + 2 \alpha  }  \\
		& \gtrsim & M^{2 + \tau - \gamma  } {1 \over (\log M)^2 } .		
\eeqa
Once again, we pigeonhole to obtain $G' \subseteq \tilde{G}$ with $|\Delta \cap (a-b + \Delta)|$ 
essentially constant when $(a,b) \in G'$. 
If the additive structure we started with has sideways comity $\nu$, then we are done; so assume not.  
This means $\gamma < \tau+ \alpha - \nu $.  Hence 
$2 + \tau - \gamma - \mu > 2- \alpha  + \nu - \mu $.  This implies that the height of 
$( G', D')$ is less than $\alpha - \nu + \mu$, which finishes the proof modulo the estimate \eqref{bigarea}.  We prove 
\eqref{bigarea} now using Cauchy-Schwarz:
\beqa
\sum _{x\in K_b} |\Delta _{x,b}| &=&
\sum _{ a\in \bigcup _{x \in K_b} \Delta _{x,b} } \sum _{x\in K_b } 
	\one _{\Delta _{x,b}} (a)  \\
& \leq & \sqrt{ |\bigcup _{x \in K_b} \Delta _{x,b} | }
	\sqrt{ \sum _a ( \sum _{x\in K_b} \one _{\Delta [x]} (a) ) ^2  }
\eeqa
We have already noted that no $a\in\Delta$ participates in more than $N^{1-\alpha }$ differences in $D$;
 i.e.,  $ \sum _{x\in K_b} \one _{\Delta [x]} (a) \lesssim N^{1-\alpha } $.  Hence the right side of the last display is less than
\beqa
 \sqrt{ | \bigcup _{x \in K_b} \Delta _{x,b} | }
\sqrt{ N^{1-\alpha  } \sum _{x\in K_b} |\Delta _{x,b}| }.
\eeqa
Rearranging terms and noting that $ \sum _{x\in K_b} |\Delta _{x,b}|  \sim |K_b| M^{\tau +\alpha} $ proves the estimate \eqref{bigarea}.  
\end{proof}

%%%%%%%%%%%%%%%%%%%%%%%%%%%%%%%%%%%%%%%%%%%%%%%%%%%%%%%%%%%%%%%%%%%%

\section{Finding structure with comity and sideways comity} \label{structuresection}

%%%%%%%%%%%%%%%%%%%%%%%%%%%%%%%%%%%%%%%%%%%%%%%%%%%%%%%%%%%%%%%%%%%%
The goal of this section is to show that when an additive structure has small comity and small sideways comity, we can find substantial additive structure in the set $\Delta$.  Precisely, we have:  

\begin{lemma} \label{sideways}
Fix $\tau > 0$.
There exists a function $f _{\tau} \colon (0, 1 ) \rightarrow (0, \infty) $ with $f_{\tau}(\eta ) \rightarrow 0$ as $\eta \rightarrow 0 $ such that the following holds.
Suppose an additive structure at height $\alpha$ has comity $\mu$, sideways comity $\nu$, $\alpha \leq {1-\tau \over 2}$,  and $E(G) \gtrsim  M^{2 + \tau}$.  Then there exists a set $H$ with 
$|H| \lesssim M^{\tau+ \alpha + f_{\tau}(\mu + \nu) }$ and 
$X$ with $|X| \lesssim M^{1 - \tau -2\alpha + f_{\tau} (\mu + \nu)  } $ such that 
\beqa
|H - H| \lesssim |H| ^{1+ f_{\tau}(\mu + \nu)}
\eeqa
\beqa
|(X+H) \cap \Delta | \gtrsim M^{1-\alpha - f_{\tau} (\mu + \nu)  }.
\eeqa
\end{lemma}
The assumption $\alpha \leq {1-\tau \over 2}$ will be valid when we use this lemma because of Proposition \ref{pigeonhole}.
We remark that the functions $f_{\tau}$ come from the asymmetric Balog-Szemeredi-Gowers theorem.  The dependence on $\tau$ comes in the ratio of the sizes of the sets $B,C$ in the statement of that theorem.  The dependence on $\mu$ and $\nu$ enter into the parameter $\eta$ in the statement of that theorem.  
%Of course this is exactly the conclusion we want:  we iterate Lemmas \ref{comitytonight} and \ref{SWcomitytonight} until we have the desired comity and sideways comity.  At that point, apply the second lemma above to obtain structure.  We first prove Lemma \ref{sideways}.

The proof of this lemma follows the idea of the proof of Lemma 6.9 in \cite{BK}.
Define
\beqa
E(A,B) = | \{ (a,b,c,d) \in A\times B\times A \times B \colon a-b=c-d \} |.
\eeqa
We will use sideways comity to obtain estimates on the quantity 
$E(\Delta [x] , F_{x,a} )$ for a typical $x\in D$ and $a\in \Delta [x]$. We show that, on average, this energy is nearly maximal; then the asymmetric Balog-Szemeredi-Gowers lemma allows us to conclude that $F_{x,a}$ is the union of translates of an almost additively closed set.
The key estimate is the following, which holds for any $x\in D$:
\begin{claim}
\beqa
  \sum _{a\in\Delta [x] }
		 \sqrt{ |\calf _x| E(\Delta [x] , F_{x,a} )  }
\geq
  \sum _{b \in \Delta }  | \Delta [x] \cap F_{x,b} | ^2
\eeqa
\end{claim}
\begin{proof}[ Proof of Lemma \ref{sideways} ]
Let's first use the claim to prove the lemma.  We sum the estimate over $x\in D$.
Note that the assumption of sideways comity is exactly what makes the right hand side large for a typical $x\in D$.  Specifically, we have 
\beqa
\sum _{x\in D}   \sum _{a\in\Delta [x] }
		 \sqrt{ |\calf _x| E(\Delta [x] , F_{x,a} )  }
& \geq &
\sum _{x\in D}  \sum _{b \in \Delta }  
		| \Delta [x] \cap F_{x,b} | ^2   \\
&\gtrsim &  { M^{3-2\alpha } M^{\tau+\alpha - \nu }  \over ( \log M )^2  }  \\
& \sim &  { M^{3+ \tau - \alpha - \nu }  \over ( \log M )^2  },
\eeqa
where the last inequality follows from estimate \ref{secondsum} and sideways comity.  Since $|D| \sim M^{2 - \tau -2\alpha}$ and $|\Delta [x]|  \sim M^{\tau+ \alpha }$ for all $x\in D$, we conclude there are $(x,a)$ in $D\times \Delta[x] $ such that 
\beqa
\sqrt{ |\calf _x| E(\Delta [x] , F_{x,a} )  } \gtrsim  M^{1 + \tau -\nu }.
\eeqa
The upper bound $|\calf_x | \lesssim M^{1-\alpha + \mu } $ allows us to conclude 
\beqa
 E(\Delta [x] , F_{x,a} ) & \gtrsim &  M^{1 + 2\tau + \alpha -2\nu -\mu}  \\
	&\sim& M^{\tau+\alpha} M^{\tau+ \alpha} M^{1-\alpha } M^{-O(\mu + \nu )}  \\
	&\sim& |\Delta [x] |^2 |F_{x,a}|  M^{-O(\mu + \nu )} .
\eeqa
We may apply the asymmetric BSG theorem to obtain the desired conclusion, namely that $\Delta [x]$ is essentially an almost additively closed set and $F_{x,a}$ is essentially a bunch of translates of $\Delta [x]$.  We remark that we use here the fact 
$|\Delta [x]| \lesssim |\calf_{x,a}|$, which gives us the right conditions for the asymmetric BSG theorem.  This fact follows from the estimate $M^{\tau + \alpha} \lesssim M^{1-\alpha + \mu}$, which holds because $\alpha \lesssim {1-\tau \over 2}$, by assumption.  This completes the proof of Lemma \ref{sideways} given the claim.
\begin{proof} [Proof of Claim]
Fix $x\in D$.  Then expand the square:
\beqa
&&\sum _{b \in \Delta }  | \Delta [x] \cap F_{x,b} | ^2  \\
&=& 
\sum _{b\in \Delta } \sum _{c\in\Delta [x]} \sum _{d\in \Delta [x] }
			 \one_{\calf _x} (c - b ) \one _{\calf _x } (d-b ) \\
&=&  \sum _{c\in\Delta [x]} \sum _{d\in \Delta [x] } 
	| \{ b\in \Delta \colon c-b \in \calf_x \text{  and  } d-b \in \calf _x \} | \\
& \leq & \sum _{c\in \Delta [x] }  
		| \{ (d, y, y') \in \Delta [x] \times \calf _x \times \calf _x \colon
		c-y = d-y' \} | \\
&=& \sum _{c \in \Delta [x] }  \sum _{y' \in \calf _x} 
			| F_{x,c} \cap (\Delta [x] - y' ) | \\
&\leq & \sum _{c \in \Delta [x] }  \sqrt{  |\calf _x| E (\Delta [x] , F_{x,c} ) }.
\eeqa
The last inequality follows from Cauchy-Schwarz.
This proves the claim, and hence the lemma. 
\end{proof}
\end{proof}

%%%%%%%%%%%%%%%%%%%%%%%%%%%%%%%%%%%%%%%%%%%%%%%%%%%%%%%%%%%%%%%%%%%%%%%%%%%

\section{Iteration} \label{iterationsection}

%%%%%%%%%%%%%%%%%%%%%%%%%%%%%%%%%%%%%%%%%%%%%%%%%%%%%%%%%%%%%%%%%%%%%%%%%%%

In this section we carry out the bookkeeping necessary for iteration of the main lemmas.  
First we iterate Lemma \ref{comitytonight} to get the following:
\begin{lemma} \label{havecomity}
Let $(G,D)$ be an additive structure at height $\alpha $ such that $E(G) \gtrsim M^{2+ \tau}$ 
and such that $E_8 (\Delta ) \lesssim M^{4+3\tau + \sigma}$.  
Then there exists an additive structure $(G',D')$ with comity $\mu = {C \over \log {1\over \sigma} }$ 
and height $\alpha ' \leq \alpha $ with $E(G')\gtrsim E(G) M^{-\mu}$.
\end{lemma}
\begin{proof}
Apply Lemma \ref{comitytonight} iteratively until we reach comity $\mu$.  Because we lower height by ${\mu \over 2} $ at each iteration, we will achieve comity $\mu$ within $\sim {1\over \mu} $ iterations.  
The energy loss after $k$ iterations is $\lesssim M^{-O(\sigma C^k)}$.  Since $k\lesssim {1\over \mu}$, we have 
$\sigma C^k \leq \sigma C^{1\over \mu} = \sigma C^{{1\over C'} \log {1\over \sigma} } << {C \over \log {1\over \sigma} }$.
This yields a structure $(G',D')$, at height $\leq \alpha$ with comity ${C \over \log {1\over \sigma} }$ and 
$E(G')\gtrsim E(G) M^{-\mu}$.
\end{proof}

Our goal is to find an additive structure on $\Delta $ with comity and sideways comity $\nu ^{\star}$ with $\nu ^{\star}$ tending to zero as the nonsmoothing parameter $\sigma_0$ tends to zero.  We of course also want this additive structure to retain most of the energy of the set $\Delta$.  With this, we can apply Lemma \ref{sideways} to obtain the additive structure we want.

By pigeonholing we can find an additive structure $(G_0, D_0)$ of height $\leq {1-\tau \over 2}$; this is Proposition \ref{pigeonhole}.  We note that the assumption on height is necessary for Lemma \ref{sideways}, and that during this iteration we will only lower the height, so the estimate on height persists.  We then iterate Lemmas \ref{havecomity} and \ref{SWcomitytonight} as follows.

Fix a parameter $\nu ^{\star}$.  This is the sideways comity we want to find.  Now we take $\sigma _0$ small enough that
\beqa
{1\over \nu ^{\star} } \sim \log \log \dots \log {1\over \sigma_{0}},
\eeqa
so that (using notation from below) $\widetilde{\mu _k} << \nu ^{\star}$ whenever $k \lesssim {1\over \nu ^{\star} }$.
The function $f$ in the statement of Theorem \ref{specstruct} is obtained by taking $\nu ^{\star} << \tau _0$ (the reason for this will be apparent at the end of this section), inverting the relationship between $\sigma_0$ and $\nu ^{\star}$, and factoring in the loss from the asymmetric Balog-Szemeredi-Gowers theorem. 

We now define a sequence of additive structures $(G_j, D_j)$ as follows:  given $(G_j, D_j)$ at height $\alpha _j$ with $E(G_j) \gtrsim M^{2+ \tau_j}$ and 
$E_8 (\Delta ) \lesssim M^{4+3\tau _j + \sigma _j}$, 
apply Lemma \ref{havecomity} to find $(\widetilde{G_{j}}, \widetilde{D_{j}})$ 
at height $\widetilde{\alpha _j} \leq \alpha _j$ with comity $\widetilde{\mu _{j} } = {C \over \log {1\over \sigma_j} }$ and 
$E(\widetilde{G_j}) \gtrsim E(G_j) M^{-\widetilde{\mu _j}}$.
If $(\widetilde{G_{j}}, \widetilde{D_{j}})$ has sideways comity $  \leq  \nu ^{\star},$
apply Lemma \ref{sideways} to obtain the desired structure.

If not, then apply Lemma \ref{SWcomitytonight} with $\sigma _{j+1} = C \widetilde{\mu _j}$ to find an additive structure $(G_{j+1} , D_{j+1} )$ of height $\alpha _{j+1} \leq \alpha _j - { \nu ^{\star} \over 2}$, such that 
$E(G_{j+1}) \gtrsim E(G_j) M^{ - \widetilde{\mu _j} }$.  
Since height drops by $ { \nu ^{\star} \over 2}$ at each iteration, we must obtain sideways comity within $\sim {1\over \nu ^{\star} }$ iterations.  Note that our estimate on the height is valid since we arranged that $\widetilde{\mu _k }<< \nu ^{\star}$ for 
$k \lesssim {1\over \nu ^{\star} }$ .

Hence there is $k\lesssim {1\over \nu ^{\star} }$ such that $(\widetilde{G_{k}}, \widetilde{D_{k}})$ has 
comity $\widetilde{\mu _k } \leq \nu ^{\star}$ and sideways comity $\nu^{\star}$, and 
\beqa
E(G_k) \gtrsim E(G_0) M^{-O(\sum _{j=0} ^{k} \widetilde{\mu _j}) } \gtrsim E(G_0) M^{-O(\widetilde{\mu _k})} \gtrsim E(\Delta) M^{-\nu ^{\star}}.
\eeqa
Hence we apply Lemma \ref{sideways} with comity $\nu ^{\star}$, sideways comity $\nu ^{\star}$, and 
$E(G_k) \gtrsim M^{2+\tau_0 - \nu ^{\star}} $.

\bigskip
\bigskip

\tiny

\textsc{M. BATEMAN, DEPARTMENT OF MATHEMATICS, UCLA, LOS ANGELES CA}

{\it bateman@math.ucla.edu}

\bigskip

\textsc{N. KATZ, DEPARTMENT OF MATHEMATICS, INDIANA UNIVERSITY, BLOOMINGTON IN}

{\it nhkatz@indiana.edu}

\end{document}